\documentclass[12pt]{article}
\usepackage{amsmath,amssymb,amsfonts,amsthm,calc}
\usepackage{eucal} 

\usepackage[all]{xy}
\usepackage{graphicx}
\usepackage{enumerate}
\usepackage{caption}
\usepackage{subcaption}

\newcounter{ecount}
\DeclareMathOperator{\Sym}{Sym}

\newtheorem{thm}{Theorem}
\newtheorem{lemma}{Lemma}
\newtheorem{prop}[lemma]{Proposition}

\newenvironment{Proof}{\par\noindent\textbf{Proof:}}
{\qed}
 \title{The Ring of Algebraic Functions on Persistence Bar Codes}
\author{Aaron Adcock  \and  Erik Carlsson  \and Gunnar Carlsson   \footnote{Research supported in part by NSF DMS-0406992} }

\begin{document}
\maketitle
\section{Introduction}
Persistent homology (\cite{topdata}, \cite{persistence})  is a fundamental tool in the area of computational topology.  It can be used to infer topological structure in data sets (see \cite{range}, \cite{patches}), but variations on the method can be applied to study aspects of the shape of point clouds which are not overtly topological (\cite{shapes}, \cite{curves}).  The methodology assigns to any finite metric space (such as are typically obtained in experimental data of various kinds) and non-negative integer $k$ a {\em bar code}, by which we will mean a finite collection of intervals with endpoints on the real line.  The integer $k$ specifies a dimension of a feature (zero-dimensional for a cluster, one-dimensional for a loop, etc.), and an interval  represents a feature which is ``born" at the  value of a parameter (the persistence parameter) given by the left hand endpoint of the interval, and which ``dies" at the value given by the right hand endpoint.  These barcodes have been demonstrated to identify structure in spaces of image patches in \cite{range} and \cite{patches}, and have been demonstrated to distinguish between handdrawn letters in \cite{curves}.  Because of the unusual structure of the invariant, i.e. as a collection of intervals rather than numerical quantities, the method currently requires substantial knowledge of topological methods.  It would clearly be useful to assign and interpret various numerical quantities attached to bar codes,  so that these outputs could be used as input to standard algorithms within machine learning, cluster analysis, and other methods.  It is the purpose of this paper to identify an algebra of functions on the set of bar codes which is defined in a conceptually coherent way. 

 The main idea is the following.  A bar code with exactly $n$ intervals can be specified by a vector $(x_1, y_1, x_2, y_2, \ldots, x_n, y_n )$, where $x_i$ denotes the left endpoint of the $i$-th interval and $y_i$ the right endpoint.  However, this representation is many to one, in that the bar code structure does not retain the ordering on the intervals.  In fact, the set of bar codes with exactly $n$ intervals can be identified with the set 
 $$ Sp^n(\Bbb{R}^2)
 $$
 the {\em $n$-fold symmetric product} of $\Bbb{R}^2$.  For any set $X$, $Sp^n(X)$ is defined to be the orbit space of the action of the symmetric group on $n$ letters on the product $X^n$ given by permuting the coordinates. 
 On the other hand, the space $(\Bbb{R}^2)^n$ is an algebraic variety over $\Bbb{R}$ (\cite{fulton}). In fact, it is an affine space of dimension $2n$, and the symmetric group action mentioned above is an algebraic action.  It is then known (see \cite{mumford}) that the orbit space inherits the structure of an algebraic variety, and the elements of its {\em affine coordinate ring} (\cite{fulton}) are functions on the set of bar codes with exactly $n$ intervals.  These affine coordinate rings are well known algebras referred to generically as rings of {\em multisymmetric polynomials} (\cite{D}).   They can be quite complicated, since it turns out that any set of algebra generators for them will satisfy non-trivial relations or {\em syzygies}.  It turns out, though, that there are inclusions of algebraic varieties
\begin{equation}\label{system}   Sp^n(\Bbb{R}^2) \rightarrow Sp^{n+1}(\Bbb{R}^2)
 \end{equation}
 which produce an inverse system  of affine coordinate rings 
 $$  \cdots \rightarrow A[Sp^{n+1}(\Bbb{R}^2)] \rightarrow A[Sp^{n }(\Bbb{R}^2)] \rightarrow \cdots
 $$ whose inverse limit we will denote, by abuse of notation, by $A[Sp^{\infty}(\Bbb{R}^2)]$.  The notation $A[-]$ denotes the affine coordinate ring.  This algebra is known to be freely generated on a set of minimal algebra generators (\cite{D}).   
 
The analysis of the system (\ref{system}) above is not sufficient, though.  This system identifies a point $((x_1, y_1), \ldots , (x_n,y_n)) \in Sp^n(\Bbb{R}^2)$ with the point $$((x_1, y_1), \ldots , (x_n,y_n),(0,0)) \in Sp^{n+1}(\Bbb{R}^2)$$  In other words, a set $S$  of $n$ intervals is identified with the set of $n+1$ intervals obtained by adjoining the interval of length zero whose two endpoints are zero.  However, in the parametrization of the isomorphism classes of persistence vector spaces in \cite{persistence} by barcodes, any interval of length zero is identified with the zero module.  So, we would like to determine the ring of all algebraic functions (i.e. the elements of $A[Sp^{\infty}(\Bbb{R}^2)]$) which have the property that they take the same value on any barcode as on the result of adjoining any interval of length zero to it.  In this paper, we will identify this subring, describe its structure, and  describe the algebra generators explicitly  so that they can be used effectively by those interested in analyzing databases of shapes.

 \section{The Ind-scheme $\frak{B}$}
 
 We first discuss the set of bar codes, without any algebraic variety structures.  For every $n$, we first consider the set of bar codes containing exactly $n$ intervals.  We will permit intervals of length zero.  The set of intervals $\frak{I}$ can be identified with the subset of   $\Bbb{R}^2$ consisting of pairs $(x,y)$ with $x \leq y$.  A bar code containing $n$ intervals is therefore identified with the $n$-fold symmetric product $Sp^n(\frak{I})$, where for any set $X$, $Sp^n(X)$ is defined to be the orbit space of the action of the symmetric group on $n$ letters on the product $X^n$ given by permuting the coordinates.  One can assemble these sets into a directed system 
  $$  \frak{I} \stackrel{i_1}{\rightarrow} Sp^2(\frak{I}) \stackrel{i_2}{\rightarrow} Sp^3(\frak{I}) \stackrel{i_3}{\rightarrow} Sp^4(\frak{I}) \rightarrow \cdots 
 $$
 where the maps $i_n: Sp^n(\frak{I}) \rightarrow Sp^{n+1}(\frak{I})$ are given by 
 $$ i_n(\{I_1, \ldots, I_n\} ) =   \{ I_1, \ldots , I_n, [0,0]\}
 $$
 The direct limit of this system will be denoted by $Sp^{\infty}(\frak{I})$.   We are interested in studying functions on $Sp^{\infty}(\frak{I})$.  Such a function can be identified with an infinite vector $(f_1, f_2, f_3, \ldots )$ of functions $f_n : Sp^n(\frak{I}) \rightarrow \Bbb{R} $ satisfying the compatibility condition 
 $$  f_{n+1} \cdot i_n = f_n
 $$
The set of all such vectors of functions forms a ring ${\cal R}$  under coordinatewise addition and multiplication.  It is not exactly what we want, however.  The reason is that under the parametrization of persistence vector spaces as described in \cite{topdata} and \cite{persistence}, intervals of length zero correspond to zero vector spaces, and therefore all intervals of length zero should be considered equal.   This means that we should consider only functions $F : Sp^{\infty}(\frak{I}) \rightarrow \Bbb{R}$ for which 
$$ F(\{ I_1, I_2, \ldots , I_n, [\xi,\xi] \}) = F(\{ I_1, I_2, \ldots , I_n, [\eta,\eta] \}) 
$$
for all possible values of $\xi$ and $\eta$.  The set of all such functions is a subring ${\cal R}^{\prime} \subseteq {\cal R}$.  This set of functions can be defined as the set of all functions on the set $\frak{B}$ defined by  
$$ \frak{B} =  \coprod _n Sp^n(\frak{I}) / \simeq
$$
where $\simeq $ is the equivalence relation generated by all  relations of the form $ \{I_1, I_2, I_n , [\xi , \xi ] \} \simeq \{ I_1, I_2, \ldots , I_n \}$.  

{\bf Remark:} The reader may suggest that one  consider instead only the subset $\frak{I}^{+}$ consisting of intervals of positive length.  This will produce a disjoint union of sets of barcodes, partitioned into the sets containing a fixed positive number of intervals of positive length.  Such a description does not take into account the fact that we would like to topologize the space of all bar codes in such a way that 
$$ lim_{\epsilon \rightarrow 0} \{ I_1, I_2, \ldots I_n, [x_{n+1}, x_{n+1} + \epsilon]\}= 
\{ I_1, I_2, \ldots , I_n \}
$$
The reason for this is that small perturbations to the input data to the persistence algorithms can modify the barcodes by modifying lengths of intervals a small amount and add intervals of small length.  This is the stability theorem for persistence diagrams proved in \cite{stability}. 

The ring of functions ${\cal R}^{\prime}$ is too large to deal with effectively. Even the much smaller ring of continuous functions on $\frak{B}$ is still too complex to describe completely.  We will observe that $\frak{B}$ is described as a colimit of algebraic varieties, and that it is therefore possible to define the {\em ring of algebraic functions on $\frak{B}$}.  It is this ring we will analyze.

Throughout this paper, $k$ will denote the field $\Bbb{R}$.  All varieties will be over $k$.  We consider the affine space $\frak{A}_n = \Bbb{A}(n)$ of dimension $2n$, parametrized with coordinates $(x_1,y_1, x_2, y_2, \ldots , x_n, y_n)$. Its affine coordinate ring is the polynomial ring $B_n = k[x_1, y_1, \ldots , x_n, y_n ]$.  There is an action of the symmetric group $S_n$ on $n$ letters on $\frak{A}_n$,  and from \cite{mumford} it follows that the set of orbits on the set of points of the variety is itself an affine algebraic variety, with affine coordinate ring equal to the invariant subring $B_n^{S_n}$.     Let $W_i \subseteq \frak{A}_n$ denote the subvariety $y _i - x_i = 0$.  We let $D_n \subseteq B_n$ denote the subring of functions whose restriction to $W_i$ is independent of $x_i$ for all $i$.  We wish to characterize this subring algebraically. 

\begin{prop} The ring $D_n$ is characterized algebraically as the subring of all $f$ for which 
$$  (\frac{\partial}{\partial x_i } + \frac{\partial}{\partial y_i })f \in (y_i -x_i)
$$
for all $i$. 
\end{prop}
\begin{Proof} We fix $i$, and consider all the functions $f$ for which $f|W_i$ is independent of $x_i$ (and therefore $y_i$).  The operator $\frac{\partial}{\partial x_i } + \frac{\partial}{\partial y_i }$ induces a differential operator on the quotient ring $Q_n  = B_n /(y_i- x_i)$, which  is identified with the partial differential operator $2 \frac{\partial}{\partial x_i}$ in
$$  Q_n  \cong k[x_1, y_1, \ldots , y_{i-1}, x_i , x_{i+1}, \ldots , x_n, y_n ]  
$$
The requirement is that the image $\overline{f}$  of $f$ in $W_i$ is independent of $x_i$, and this is equivalent to the condition $ \frac{\partial}{\partial x_i}(\overline{f}) = 0$.  This condition is to hold for each $i$, which gives the result. 
\end{Proof}

 \section{The ring of algebraic functions on $\frak{B}$}
 
We begin by changing coordinates via the formulae
  $\xi_i =   x_i +y_i$ and $\eta _i = y_i -x_i$.  It is clear that $B_n$ can also be identified with $k[\xi _1, \eta _1, \ldots, \xi _n, \eta _n ]$, and that the symmetric group in the new coordinate system  permutes the $\xi _i$'s and $\eta _i $'s.  Under this transformation, the operator $\frac{\partial}{\partial x_i } + \frac{\partial}{\partial y_i }$ is carried into the operator $2 \frac{\partial}{\partial \xi _i}$.  This means that the ring $D_n$ is identified with the subring of functions $f(\xi _1, \eta _1,  \ldots , \xi _n, \eta _n )$ for which $\frac{\partial f}{\partial \xi _i} \in (\eta _i)$ for all $i$.

\begin{prop} \label{monomials}  A $k$-basis for the $D_n$ is given by the set of monomials 
$$\xi _1^{a_1}\xi _2 ^{a_2}\cdots \xi _n^{a_n} \eta _1^{b_1} \eta _2 ^{b_2} \cdots \eta _n^{b_n}$$
 for which $a_i >0$ implies $b_i > 0$.  
\end{prop} 
\begin{Proof} We note that the operator $\partial / \partial \xi _i$ carries each monomial to a constant multiple  of a single monomial, namely the monomial obtained by decreasing $a_i$ by one.  Moreover, containment in the ideal $(\eta _i)$ is also given purely by conditions on monomials, i.e. that $b_i > 0$.  We conclude that $D_n$  is spanned by monomials lying in $D_n$.  But it is clear that a monomial $\mu$  lies in $D_n$ exactly if it is the case that whenever $\xi_i$ divides $\mu$, then $\eta _i$ also divides $\mu$.  This  corresponds to the above numerical condition on the exponents in the monomial.   \end{Proof} 

 The symmetric group  action  clearly preserves the subring $D_n$.  Moreover, it preserves the basis of monomials within $D_n$.  Let $\{ \mu _{\alpha}\}_{\alpha \in A} $ denote a set of orbit representatives of the $S_n$-action on the set of monomials defined in Proposition \ref{monomials}.  Let $\sigma _{\alpha}$ denote the sum of all the elements in the orbit of $\mu _{\alpha}$.  

\begin{prop} \label{count} We let $D_n^{S_n}$ denote the subring of elements of $D_n$ which are invariant under the action of $S_n$.  Then the elements 
$\sigma _{\alpha}$ form a $k$-basis of $D_n^{S_n}$.  
\end{prop} 
\begin{Proof}
This result plainly holds for any algebra over a field of characteristic zero on  which there is a $G$-action which preserves a basis of monomials.  
\end{Proof} 

We have restriction maps $\pi _{n,m}: D_n \rightarrow D_m$, when $n \geq m$, defined by $\pi _{n,m}(\xi _i (\mbox{resp } \eta _i)) = \xi _i(\mbox{resp } \eta _i)$ for $i \leq m$, and $\pi _{n,m} (\xi _i) = 0 $ for $i > m$.  The map $\pi _{n,m}$ is $S_m$-equivariant, where $S_m$ acts by permuting the first $m$ pairs of variables. It follows that we may construct composites 
$$   D_n^{S_n} \hookrightarrow D_n^{S_m} \stackrel{\pi_{n,m}^{S_{m}} }\rightarrow D_m ^{S_m} 
$$
which we denote by $\sigma _{n,m}$, and therefore the inverse system 

$$  \cdots\stackrel{\sigma  _{n+1,n}}{\longrightarrow} D_n^{S_n} \stackrel{\sigma  _{n,n-1}}{\longrightarrow} D_{n-1}^{S_{n-1}}\stackrel{\sigma  _{n-1,n-2}}{\longrightarrow} \cdots  
\stackrel{\sigma  _{2,1}}{\longrightarrow} D_1
$$
We will denote the inverse limit of this system by $\frak{D}$.

We next recall some of the notation and basic facts about multisymmetric polynomials, which can be found in Dalbec \cite{D}.
Let $R_{n,r}$ be the polynomial ring in $nr$ variables,
\[R_{n,r} = k[x_{i,j};1 \leq i \leq n, 1 \leq j \leq r]\]
We let the symmetric group $S_n$ act on $R_{n,r}$ via the formula $\sigma(x_{ij}) = x_{\sigma(i)j}$, and let 
\[\Lambda_{n,r} = R_{n,r}^{S_n},\]
denote the ring of $S_n$ invariants. 
There is an inverse system parallel to the one constructed above involving the rings $\Lambda_{n,r}$. We have  evaluation maps
\[\pi_{n,m} : R_{n,r} \rightarrow R_{m,r},\quad m \leq n\]
defined by setting $x_{ir}=0$ if $i>m$.   The map $\pi _{n,m}$ is $S_m$-equivariant, when $S_m \subseteq S_n$ is the subgroup of permutations of the first $m$ elements of the set $\{ 1, \ldots , m \}$.  We have the composites 

$$ \Lambda _{n,r} = R_{n,r}^{S_n} \hookrightarrow R_{n,r}^{S_m} \stackrel{\pi_{n,m}^{S_{m}} }\rightarrow R_{m,r}^{S_m} = \Lambda _{m,r}
$$
which we denote by $\rho _{n,m}$.  
The inverse limit of the system 
$$  \cdots\stackrel{\rho _{n+1,n}}{\longrightarrow} \Lambda _{n,r} \stackrel{\rho _{n,n-1}}{\longrightarrow} \Lambda _{n-1, r} \stackrel{\rho _{n-1,n-2}}{\longrightarrow} \cdots  
\stackrel{\rho _{2,1}}{\longrightarrow} \Lambda _{1,r}
$$
will be denoted by $\Lambda _r$, and referred to as the ring of {\em r-multisymmetric} functions.  
 It has a grading
\[\Lambda_{r} = \bigoplus_k \Lambda^k_r\]
induced by the grading on $R_{n,r}$.  There is an evident embedding $\frak{D} \hookrightarrow \Lambda _2$. We will use this embedding to identify the structure of $\frak{D}$.   

The ring of multisymmetric functions has several interesting sets of generators. Given an array of nonnegative integers 

$$ \left[ \begin{array}{cccc}
 a_{11}  & a_{12}  & \cdots & a_{1r}  \\
a_{21}  & a_{22}  &  \cdots & a_{2r}  \\
 \vdots   &  \vdots  &  \ddots   &  \vdots \\
 a_{k1}   & a_{k2}  & \cdots  &  a_{kr}   
\end{array} 
\right]
$$
with $k \leq n$ and 
${\bf a}_i = (a_{i1},...,a_{ir}) 
$,
we define the multisymmetric monomials by
\[m_{{\bf a}_1,...,{\bf a}_k} = \Sym (x_{11}^{a_{11}}\cdots x_{kr}^{a_{kr}}) = \sum _{\sigma \in S_n} \prod _{i,j } x_{{\sigma(i)}j }^{a_{ij}}\in \Lambda_{r},\]
 $\Sym$ applied to  a monomial yields the sum of all monomials which are in the orbit of the $S_n$-action.  

They form a vector space basis of $\Lambda_{n,r}$, for any $n$.
It is known that $\Lambda_{n,r}$ is generated as an algebra by the symmetrizations of  monomials involving only $\{ x_{11}, x_{12}, \ldots , x_{1r} \} $.  They are given by the formulae
\[p_{{\bf a}} = m_{{\bf a}} = \sum_i x_{i1}^{a_1}\cdots x_{ir}^{a_r},\]
and are called the multisymmetric \emph{power sums}.
While there are relations among the power sums in finitely many variables, they freely generate the inverse limit $\Lambda_{r}$,
making it a polynomial algebra. See \cite{D} for details.

Our interest is  in  the case $r=2$. Let us set
\[\xi _i = x_{i1},\quad \eta _i = x_{i2},\]
and as before
\[B_n = R_{n,2} = k[\xi _1, \eta _1,...,\xi _n,\eta _n].\]
%
%
%
%
%
%
The subalgebra $\frak{D} \subseteq \Lambda _2$ now has the following characterization. 
\begin{thm}
As a subalgebra of $\Lambda_2$, $\frak{D}$ is freely generated by the set $\Delta$ of elements of the form $
p_{a,b}$ where 
$b \geq 1$.
\end{thm}
\begin{Proof}  We first consider the subalgebra $k[\Delta ] \subseteq \frak{D}$ generated by $\Delta$. Because $\Delta$ is a subset of the free generating set of $\Lambda _2$, it is clear that the composite 
$$ k[\Delta ] \rightarrow \frak{D} \hookrightarrow \Lambda _2
$$
is injective and isomorphic onto a polynomial subalgebra, and therefore that $k[\Delta ]$ is itself a polynomial subalgebra of $\frak{D}$.  In order to prove that $k[\Delta ] = \frak{D}$, we only need to count dimensions, and we formulate the counting in terms of the Hilbert series.  Recall that for a graded $k$-vector space $V_*$, we have the Hilbert series 
$$ P(V_*) = \sum _i \mbox{dim}_k(V_i )t^i
$$
The Hilbert series for a polynomial algebra on a single generator $x$ of grading $i$ is $(1-t^i)^{-1}$. Moreover, if we are given two graded vector spaces $V_*$ and $W_*$, then $P(V_* \otimes W_*) = P(V_*)P(W_*)$.  Since there are $n$ monomials of degree $n$ in $\Delta$, we find that the Hilbert series for $k[\Delta]$ is 
$$ P(k[\Delta])= \prod _n (1-t^n)^{-n}
$$
If we can show that the Hilbert series for $\frak{D}$ is equal to this series, the proof will be complete.  

In Proposition \ref{count}, we found that a $k$-basis for $\frak{D}$ may be identified with a set of orbit representatives of the $S_n$-action on the set of all monomials which have the property that if the exponent of $\xi _i$ is non-zero, then so is the exponent of $\eta _i$.   Such a set of representatives is given by  the set of monomials of the form
\[\xi _1^{a_1}\eta _1^{b_1} \cdots \xi _l^{a_l}\eta _l^{b_l},\quad  \varphi^{-1}(a_i,b_i) \geq \varphi^{-1}(a_{i+1},b_{i+1})\]
where $l \leq n$, and $\varphi : \mathbb{N}^+ \rightarrow \mathbb{N}^+\times \mathbb{N} $ is the bijection
\[\left(\varphi_1,\varphi_2,...\right) = \left((1,0),(1,1),(2,0),(1,2),(2,1),(3,0),(1,3),...\right)\]
onto the set of possible nonzero exponents.
The dimension of the $k$-graded component of $B^{S_n}_n$ is just the number of these monomials of degree $k$.
Let us say that $(a,b) \leq (c,d)$ when $\varphi^{-1}(a,b) \leq \varphi^{-1}(c,d)$, and let $f(a,b,k)$ denote the number of sequences $(a_1,b_1,...,a_l,b_l)$ such that
\[(a,b) \geq (a_1,b_1) \geq \cdots \geq (a_l,b_l),\quad (a_i,b_i) \in \mathbb{N}^+\times \mathbb{N}\]
and 
$$ \sum _1^l (a_i+b_i) = k
$$
with no restrictions on $l$. It is easy to check that it satisfies the recursion relation
\[f(a,b,k) = \sum_{(c,d) \leq (a,b)} f(c,d,k-c-d),\]
This corresponds to a rule for the generating function $f_{a,b}(t) = \sum _k f(a,b,k) t^k$,
\[f_{a,b}(t)  = 1+\sum_{(c,d) \leq (a,b)} t^{c+d}f_{c,d}(t) = 1+\sum_{(c,d) < (a,b)} t^{c+d}f_{c,d}(t) + t^{a+b} f_{a,b}(t)\]
Solving for $f_{a,b}(t)$ gives
\[(1-t^{a+b})f_{a,b}(t) = 1+\sum_{(c,d) < (a,b)} f_{c,d}(t) = f_{a',b'}(t),\]
where $(a',b')$ is the element immediately below $(a,b)$ under $\varphi$.
It is readily verified that the formula 
\[g_{a,b}(t) = (1-t^{a+b})^{-a}\prod_{1 \leq k \leq a+b-1} (1-t^k)^{-k},\]
satisfies the same recursion relation, and therefore that $f_{a,b}(t) = g_{a,b}(t)$.  Taking limits 
\[\lim_{n \rightarrow \infty} P(B^{S_n}_n) = \lim_{(a,b) \rightarrow \infty} \sum_{k\geq 1} f(a,b,k)t^k = \prod_{k \geq 1} (1-t^k)^{-k}\]
gives the result.

\end{Proof}

\section{Machine Learning on $\frak{B}$ with examples}
\subsection{Digits Example}
To illustrate the classification potential of this technique, we apply it to the MNIST database \cite{MNIST}, of handwritten digits.  We emphasize that the aim is not to outperform existing machine learning algorithms for digit classification, but to present an example that demonstrates one way of combining this technique with existing machine learning techniques.  While it is clear that pure topological classification cannot distinguish between the digits (there are three numbers that do not have any loops, three that always have loops, one that has two loops and three that have style-dependent loops), we can use the power of persistent homology to sift out more information.  We begin by showing the full analysis of a few digits and then give the empirical results of applying this technique to a subset of the MNIST database.

\subsubsection{Topological Methods}
\label{sec:digtop}
We begin by describing a particular graph construction given a digital image.  We treat the pixels as vertices and add edges between adjacent pixels (including diagonals).  We can now define a filtration on the vertices of the graph corresponding to the image pixels.  A natural filtration could be constructed using the pixel intensities of the original image (see Figure \ref{fig:filt}, Section \ref{sec:lesion}).  Another filtration, used in \cite{curves}, can be constructed by thresholding, to produce a binary image, and adding 1-pixels as we sweep across the image.  This adds spatial information into what would otherwise be a purely topological measurement.  Since the orientation of the digit matters (a 6 is the same as a 9 given a 180 degree rotation), we choose the latter approach and sweep across the rows and columns of each digit.

By taking into account spatial information, we get a rough view of the location of various topological features.  For example, though a `9' and `6' both have one connected component and a single loop, the loop will appear at different locations in the top-down filtration for the `9' and `6'.  The digits and one of the resulting barcodes are shown in Figures \ref{fig:no_loop} and \ref{fig:loop}.  Using all four sweeps, and both the Betti 0 and Betti 1 barcodes, reveals additional differences between each of the digits.

\begin{figure}[h!]
\begin{subfigure}{0.15\textwidth}
\includegraphics[width=0.5in]{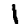}
\end{subfigure}
\begin{subfigure}{0.3\textwidth}
\includegraphics[width=5in]{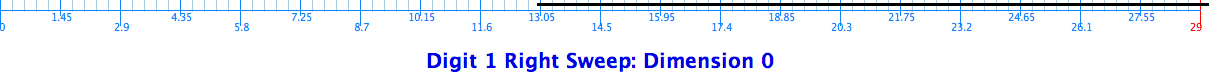}
\end{subfigure}

\begin{subfigure}{0.15\textwidth}
\includegraphics[width=0.5in]{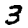}
\end{subfigure}
\begin{subfigure}{0.3\textwidth}
\includegraphics[width=5in]{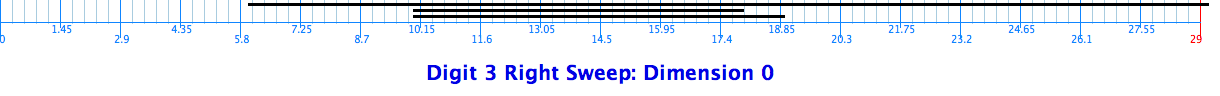}
\end{subfigure}

\begin{subfigure}{0.15\textwidth}
\includegraphics[width=0.5in]{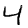}
\end{subfigure}
\begin{subfigure}{0.3\textwidth}
\includegraphics[width=5in]{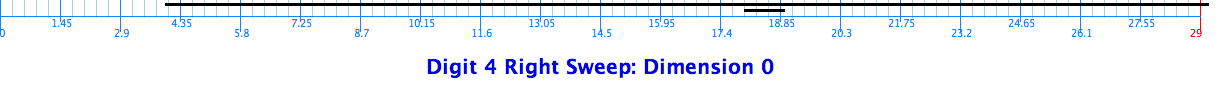}
\end{subfigure}

\begin{subfigure}{0.15\textwidth}
\includegraphics[width=0.5in]{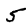}
\end{subfigure}
\begin{subfigure}{0.3\textwidth}
\includegraphics[width=5in]{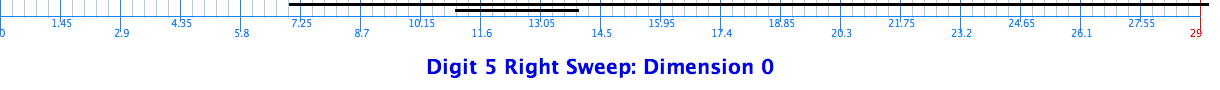}
\end{subfigure}

\begin{subfigure}{0.15\textwidth}
\includegraphics[width=0.5in]{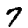}
\end{subfigure}
\begin{subfigure}{0.3\textwidth}
\includegraphics[width=5in]{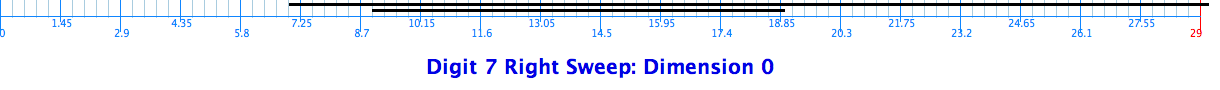}
\end{subfigure}

\caption{No Loop Digits with Betti 0 barcode, sweep to right}
\label{fig:no_loop}
\end{figure}

\begin{figure}[h!]
\begin{subfigure}{0.15\textwidth}
\includegraphics[width=0.5in]{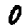}
\end{subfigure}
\begin{subfigure}{0.15\textwidth}
\includegraphics[width=5in]{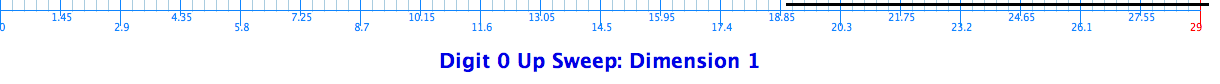}
\end{subfigure}

\begin{subfigure}{0.15\textwidth}
\includegraphics[width=0.5in]{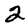}
\end{subfigure}
\begin{subfigure}{0.15\textwidth}
\includegraphics[width=5in]{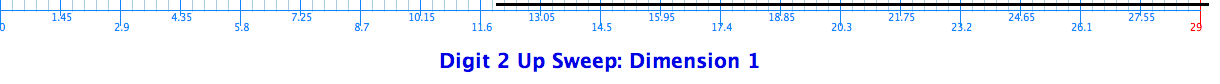}
\end{subfigure}

\begin{subfigure}{0.15\textwidth}
\includegraphics[width=0.5in]{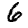}
\end{subfigure}
\begin{subfigure}{0.3\textwidth}
\includegraphics[width=5in]{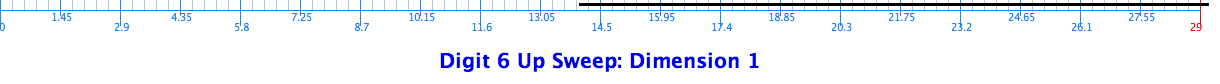}
\end{subfigure}

\begin{subfigure}{0.15\textwidth}
\includegraphics[width=0.5in]{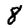}
\end{subfigure}
\begin{subfigure}{0.3\textwidth}
\includegraphics[width=5in]{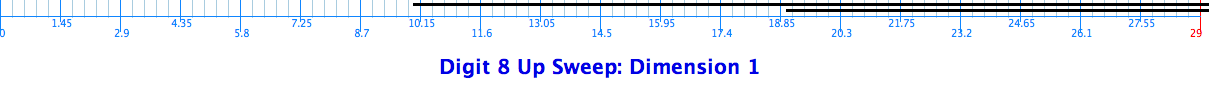}
\end{subfigure}

\begin{subfigure}{0.15\textwidth}
\includegraphics[width=0.5in]{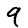}
\end{subfigure}
\begin{subfigure}{0.3\textwidth}
\includegraphics[width=5in]{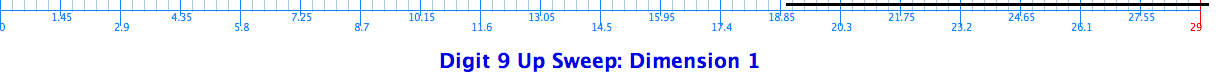}
\end{subfigure}

\caption{Loop Digits with Betti 1 barcode, sweep to top}
\label{fig:loop}
\end{figure}

\subsubsection{Feature Selection}

We can use the techniques described in this paper to coordinatize the barcode space $\frak{B}$.  In machine learning terminology, these coordinates are called \emph{features}.  This allows us to characterize the barcodes generated by each data point as a compact feature vector.  This also gives us great flexibility in selecting features that work well with our data.  We can then apply a standard machine learning algorithm, such as a support vector machine (SVM), to classify the data.

We selected a set of four features from the invariants discussed in this paper.  Intuitively, the exponents in each polynomial will give the relative value of small bars or endpoints compared to large bars or endpoints.  For example, if comparing two bars of length $\frac{b}{2}$ and $b$, the first bar will have more weight in an invariant linear polynomial than in an invariant quadratic polynomial.  Indeed, 
\begin{align*}
\left(\frac{b}{2}\right)^2 &= \frac{b^2}{4}, \\
\left(\frac{b}{2}\right)^3 &= \frac{b^3}{8}, \\
\left(\frac{b}{2}\right)^4 &= \frac{b^4}{16}, \\
&\mathrel{\makebox[\widthof{=}]{\vdots}}
\end{align*}

We selected four features,

\begin{align*}
\sum_{i} &x_{i}(y_{i}-x_{i}) \\
\sum_{i} &(y_{max}-y_{i})(y_{i}-x_{i}) \\
\sum_{i} &x_{i}^{2}(y_{i}-x_{i})^4 \\
\sum_{i} &(y_{max}-y_{i})^2(y_{i}-x_{i})^4 
\end{align*}
which when applied to the four sweeps, each with a 0-dimensional and 1-dimensional barcode, gives a feature vector of total size 32 which we then arranged into a feature matrix.   Intuitively speaking, the first two features take all of the bars, lengths and endpoints, into account.  The second two features heavily favor the arrangement of longer bars.    A visualization of a matrix of 10,000 digits using classical multidimensional scaling (MDS) is shown in Figure \ref{fig:viz} and the spectrum of the matrix is shown in Figure \ref{fig:lab}.

\begin{figure}
\begin{subfigure}{\textwidth}
\centering
\includegraphics[width=5in]{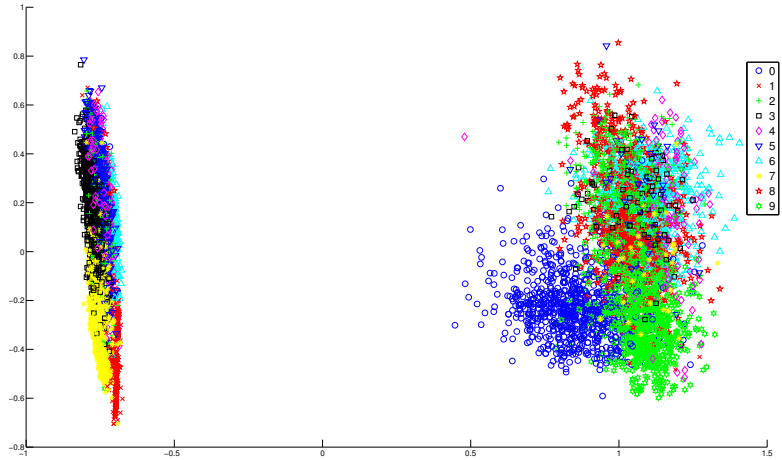}
\caption{A 2D View of the Data}
\end{subfigure}

\begin{subfigure}{\textwidth}
\centering
\includegraphics[width=5in]{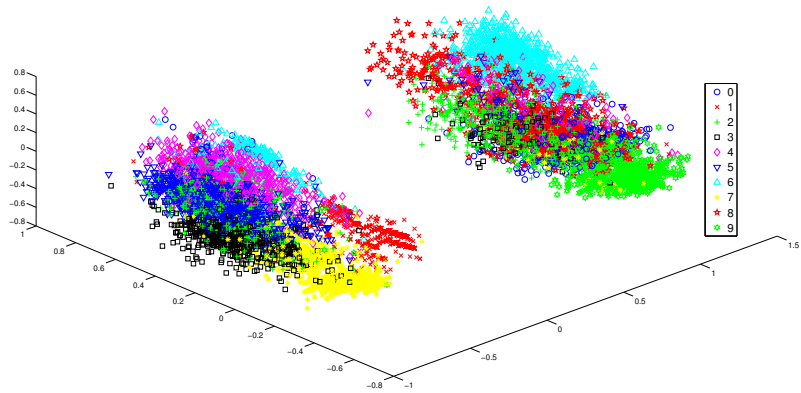}
\caption{A 3D View of the Data}
\end{subfigure}
\caption{Visualization of Data using Topological Features}
\label{fig:viz}
\end{figure}

\begin{figure}[h]
\centering
\includegraphics[width=5in]{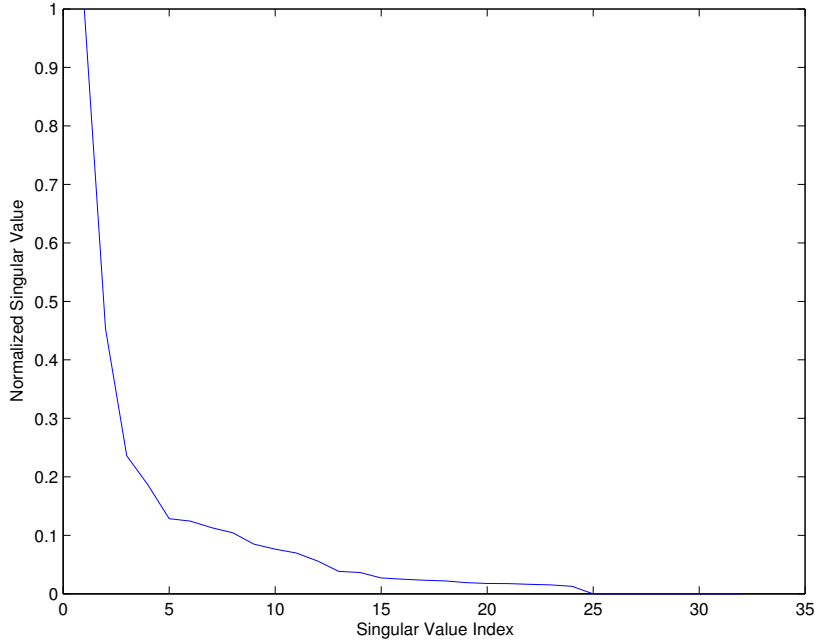}
\caption{Normalized Spectrum of Topological Feature Matrix}
\label{fig:lab}
\end{figure}

As is typical when using a SVM, we scaled each coordinate such that the values were between 0 and 1.  The SVM was implemented using software provided by \cite{libsvm}.

\subsubsection{Classification Results}
We applied these methods on a subset of 1000 digits from the MNIST database to tune parameters of the algorithm and test various kernels.  For the radial basis function $e^{-\gamma |u-v|^2}$ (RBF, also known as the Gaussian kernel),  we used $\gamma =8$. For the polynomial kernel $(\gamma(u*v) + a)^d$, we used $d=3$ with $\gamma = 2$ and $a=2$.  In both functions, $u$ and $v$ represent the calculated feature vectors.  After this, we progressively increased the size of the subset to 10,000 handwritten digits.

The classification accuracy was measured by partitioning the data set into one hundred subsets and using cross-validation successively on each subset.  The results are shown in Table \ref{tbl:SVM}.

\begin{table}[h]
\caption{Classification Accuracy of two SVM Kernels}
\label{tbl:SVM} 
\begin{center}
\begin{tabular}[h]{| l | c | c | c |}
\hline
SVM & 1000 Digits & 5000 Digits & 10000 Digits \\ \hline
Gaussian & 87.70\% & 91.54\% & 92.04\% \\ \hline
Polynomial & 88.00\% & 91.62\% & 92.10\% \\ \hline
\end{tabular}
\end{center}
\end{table}

With the polynomial kernel, an error of $7.9\%$ is seen.   As mentioned above, the purpose of this test is not to outperform existing classification algorithms but to demonstrate one application of the topological features.  In line with this, we examined some of the digits that the algorithm failed on.  Figure \ref{fig:prob} shows a few of the typical problem digits.  

\begin{figure}[h!]
\begin{subfigure}{\textwidth}
\centering
\includegraphics[width=0.9in]{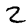}
\includegraphics[width=0.9in]{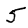}
\includegraphics[width=0.9in]{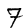}
\includegraphics[width=0.9in]{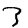}
\includegraphics[width=0.9in]{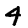}
\caption{Stylistic Problems}
\end{subfigure}

\begin{subfigure}{\textwidth}
\centering
\includegraphics[width=0.9in]{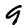}
\includegraphics[width=0.9in]{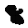}
\includegraphics[width=0.9in]{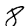}
\includegraphics[width=0.9in]{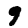}
\includegraphics[width=0.9in]{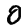}
\caption{Spurious Topological Changes}
\end{subfigure}

\caption{Common Misclassifications}
\label{fig:prob}
\end{figure}

The most common confusion is between a `5' and a `2' written with no loop. Other confusions often occur between the shown style of `7' and slanted `3's and between a certain style of `4' and a `9'.  These confusions are not unexpected since these numbers are topologically the same.  The extra spatial information added by the directional sweeps is sensitive to variations in the slant or style of handwriting and a visual inspection of these digits suggests why the algorithm has difficulty classifying these particular examples.  Other common confusions occur when topological changes occurred to the digit, specifically when the writer adds or removes a loop.

\subsection{Hepatic Lesion Classification}
\label{sec:lesion}
In this example, we apply topological features to classifying hepatic lesions.  The dataset consists of computed tomography (CT) scans of 132 hepatic lesions that are outlined and annotated by radiologists.  There are nine diagnoses represented in the data: cysts (45 lesions), metastases (45 lesions), hemangiomas (18 lesions), hepatocellular carcinomas (HCC, 11 lesions), focal nodules (5 lesions), abscesses (3 lesions), neuroendocrine neoplasms (NeN, 3 lesions), a single laceration and a single fat deposit.  Additionally, there are no controls for the size of the lesion and the lesions vary from under 100 pixels to 10,000 pixels.  Because of the unbalanced nature of the data, we focus on the subset of cysts, metastases, and hemangiomas.

Classification results using the barcode metric (matching metric) were first presented in \cite{metric lesion}, and we follow the same methods for processing and generating barcodes from the data.  We will briefly describe the methods here.  For a more detailed account, please read \cite{metric lesion}.

\subsubsection{Topological Methods}
As mentioned above, a natural filtration for an image is to filter by the pixel intensity.  An example of this filtration is given in Figure \ref{fig:filt}.  The variation in pixel intensity allows us to use a one-dimensional filtration on the pixel intensity, but as the results will show, the classification is improved when geometric information is added into the filtrations.

As there is no rotational orientation of the lesions, we cannot add in geometric information using the sweeps described in the previous section.  Instead, we use the lesion border provided by the radiologist and assign each pixel its distance from the border.  Then, by using two-dimensional homology, we achieve improved results, especially in the case of the hemangiomas which are characterized by large dense regions on the outer part of the lesion.  Because two-dimensional filtrations are computationally intensive, we approximate the two-dimensional filtration with one-dimensional barcode `slices' along the border filtration axis.  We use 7 slices per lesion and both the Betti 0 and Betti 1 barcodes.

Note that we can look at each filtration from each direction and catch different features.  The intensity filtration can add high intensity pixels first or low intensity pixels first.  The boundary filtration can begin with pixels near the boundary first or pixels far from the boundary first.  This yields 56 one-dimensional barcodes per lesion.

\begin{figure}[!ht]

\begin{subfigure}{\textwidth}
\center
\includegraphics[width=4.75in]{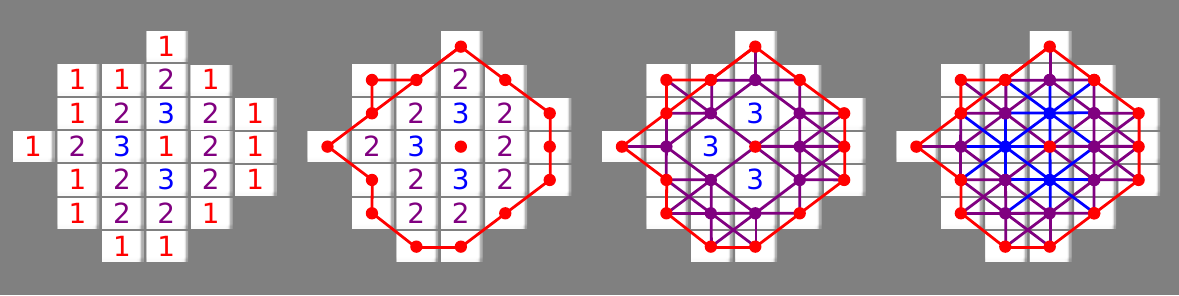}
\caption{Simple image with filtered complex}
\label{fig:toyim}
\end{subfigure}
\begin{subfigure}{\textwidth}
\center
\includegraphics[width=4.0in, trim=0in .5in 0in 0in, clip]{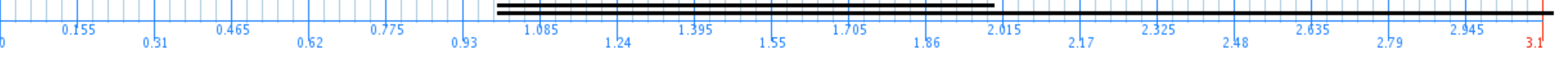}
\caption{$\beta_0$ barcode for above image}
\end{subfigure}
\begin{subfigure}{\textwidth}
\center
\includegraphics[width=4.0in, trim=0in .5in 0in 0in, clip]{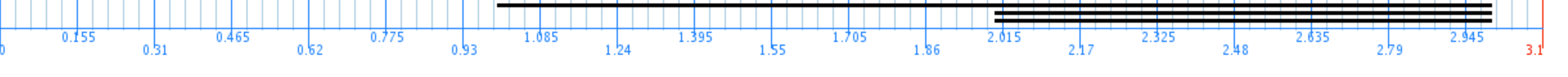}
\caption{$\beta_1$ barcode for above image}
\end{subfigure}
\caption{Constructing an increasing 1D-filtration on an image \cite{metric lesion}}
\label{fig:filt}
\end{figure}

\subsubsection{Feature Selection}
We use a slightly different set of four features as compared to the digits example.  These features are shown below.  The two sets of features that focus on long bars and features which take into account shorter bars is used here.  In this application, this is analogous to filtering the barcode to remove the large number of smaller bars.  Because of the variations in lesion size, we look at the average over each bar to try and eliminate the effects of large variations in lesion size.

\begin{align*}
\sum_{i}^n &x_{i}(y_{i}-x_{i})/n \\
\sum_{i}^n &(y_{max}-y_{i})(y_{i}-x_{i})/n \\
\sum_{i}^n &x_{i}^{2}(y_{i}-x_{i})^4/n \\
\sum_{i}^n &(y_{max}-y_{i})^2(y_{i}-x_{i})^4/n \\
\end{align*}

As mentioned above, we have 56 barcodes per lesion.  With four features, this yields a feature vector of 224 features for each lesion.


\subsubsection{Classification Results}
We apply the SVM using only the Gaussian kernel and use an exponential parameter sweep to find optimal values of $\gamma$ for each method.  We use LOOCV to calculate the classification accuracies.  The results are shown below.  Table \ref{tbl:SVM} gives the results for 1D and 2D filtrations for several different datasets while Table \ref{tbl:lesion_miss} shows how well the algorithm performs on different lesion types for the different filtrations.  Table \ref{tbl:size_class} demonstrates the effect of size on classification.

\begin{table}[h]
\caption{SVM Classification Accuracies for 1D and 2D Filtrations}
\label{tbl:SVM} 
\begin{center}
\begin{tabular}[h]{| l | c | c | c | c |}
\hline
Filtration & Full  & HcHeCM & HeCM & CM \\ \hline
1D (Intensity) & 53.03\% & 59.66\% & 65.74\% & 75.56\% \\ \hline
2D & 67.42\% & 74.79 \% & 81.48\% & 86.67\% \\ \hline
\end{tabular}
\end{center}
\end{table}

\begin{table}[h]
\caption{HeCM \% Classification Accuracy by Lesion Type}
\label{tbl:lesion_miss} 
\begin{center}
\begin{tabular}[h]{| c | c | c | c | c |}
\hline
Filtration & \% of HeCM & \% of Heman.  & \% of Cysts & \% of Metas. \\ \hline
1D & 65.74\% & 33.33\% & 75.56\% & 68.89\% \\ \hline
2D & 81.48\% & 61.11\% & 86.67\% & 84.44\% \\ \hline
\end{tabular}
\end{center}
\end{table}

\begin{table}[h]
\caption{Classification by Lesion Size of HeCM}
\label{tbl:size_class} 
\begin{center}
\begin{tabular}[h]{| c | c | c | c | c |}
\hline
Lesion Size by Area & \% Accu. & \# of Heman. & \# of Cysts & \# of Metas. \\ \hline
All & 81.48\% & 18 & 45 & 45 \\ \hline
$<$10000 px & 82.52\% & 18 & 42 & 43 \\ \hline
$<$5000 px  & 84.78\% & 16 & 39 & 37 \\ \hline
$<$2500 px  & 86.25\% & 14 & 32 & 34 \\ \hline
$<$1250 px  & 88.514\% & 8 & 28 & 23 \\ \hline
\end{tabular}
\end{center}
\end{table}

Using \cite{metric lesion}, we see that that topological features are comparable with using the matching metric to generate features.  The results from the HeCM dataset for the two methods are shown below.  They reflect the correct classification of a single lesion using a the topological features, making the two methods virtually the same for this subset of the data.  Comparing with the other results in \cite{metric lesion} shows that the two results are very close in most categories, with each slightly outperforming the other in certain subsets of the data.

\begin{table}[!h]
\caption{Classification Methods}
\label{tbl:metric} 
\begin{center}
\begin{tabular}[h]{| c | c | c |}
\hline
Filtration & Barcode Features & Matching Metric \\ \hline
1D & 65.74\% & 63.80\% \\ \hline
2D & 81.48\% & 80.56\% \\ \hline
\end{tabular}
\end{center}
\end{table}

\subsection{Discussion}
These two examples demonstrate the classifying power of topological features when applied to real world datasets.  This was done using off-the-shelf machine learning algorithms showing that these features can easily be combined with more traditional classification methods adding a set of additional classification features to the machine learning toolbox.

These examples also show the power of combining topology with geometry.  In both datasets, this is an integral part of the classification procedure.  The results in the hepatic lesion dataset provide an especially good example of the potential gains that can be achieved by combining both fields.

In summary, using algebraic geometry and invariant theory, we have identified a family of coordinates on the space of finite metric spaces, or sampled shapes.  These coordinates can serve as a method for organizing the collection of all barcodes, and therefore any database whose members produce barcodes.   Of course, we can also use various metrics on barcode space, such as the bottleneck or Wasserstein distances.  It would be extremely interesting to   analyze the relationship between these distances on barcode spaces with various more algebraic notions of distance on the barcode coordinates.  It would also be very interesting to define and analyze analogous coordinates on spaces of multidimensional persistence modules, where they might give information which is currently not accessible due to the complexity of the algebraic descriptions of multidimensional persistence modules.

\end{document}